\documentclass[11pt,letterpaper]{amsart}
\usepackage{macrosS20} 
\usepackage{appendix}
\usepackage{ mathrsfs }
\usepackage{xcolor}

\newcommand{\m}{\mathcal}

\newcommand{\dd}{{\mathrm{d}}}

\newcommand{\sP}{{\mathscr P}}

\DeclareMathOperator{\tr}{tr}
\DeclareMathOperator{\Real}{Re}

\numberwithin{equation}{section}

\begin{document}

\title[Hidden monotonicity and canonical transformations for MFG]{Hidden monotonicity and canonical transformations for mean field games and master equations}

\author[M. Bansil]{Mohit Bansil}
\address{Department of Mathematics, University of California, Los Angeles, CA 90095-1555, USA}
\email{mbansil@math.ucla.edu} 

\author[A.R. M\'esz\'aros]{Alp\'ar R. M\'esz\'aros}  
\date{\today}
\address{Department of Mathematical Sciences, University of Durham, Durham DH1 3LE, England}
\email{alpar.r.meszaros@durham.ac.uk} 

\maketitle

\begin{abstract}
In this paper we unveil novel monotonicity conditions applicable for Mean Field Games through the exploration of finite dimensional {\it canonical transformations}. Our findings contribute to establishing new global well-posedness results for the associated master equations, also in the case of potentially degenerate idiosyncratic noise. Additionally, we show that recent advancements in global well-posedness results, specifically those related to displacement semi-monotone and anti-monotone data, can be easily obtained as a consequence of our main results.
\end{abstract}

\section{Introduction}

Mean field games (MFGs for short) have been introduced in the pioneering works of Lasry--Lions and Huang--Malham\'e--Caines (see \cite{LasLio:07, HuaMalCai}). The main motivation of both groups was to model strategic decision making in systems involving a large number of rational agents, arising from (stochastic) differential games. Ever since, this theory witnessed a great success, both from the theoretical viewpoint and the point of view of applications. We refer to \cite{CD1, Carmona2018, CarPor} for a thorough, relatively up to date description of the evolution of this field, from the probabilistic and analytic aspects.

\medskip

Already early on, Lions in his lecture series at Coll\`ege de France (\cite{lions07}) has introduced the so-called {\it master equation}, associated to MFGs. This is a nonlocal and nonlinear PDE of hyperbolic type set on $\R^d\times\sP_2(\R^d)$, where $\R^d$ models the state space of a typical agent, while $\sP_2(\R^d)$ (the set of Borel probability measures with finite second moment, supported on $\R^d$) encodes the distribution of the agents. One of the main motivations for the solvability of the master equation is that it provides a deep link between games with finite, but large number of agents and the corresponding MFG: classical solutions to the master equation serve as great tools to obtain quantitative rates of convergence of closed loop Nash equilibria of games with finite number of agents, when the number of agents tends to infinity.

\medskip

The master equation that we consider in this paper writes as follows. As data, we are given a Hamiltonian $H:\R^d\times\sP_2(\R^d)\times\R^d\to\R$  and a final cost $G:\R^d\times\sP_2(\R^d)\to\R$. We emphasize that throughout the text we assume that $H$ and $G$ are smooth enough (we detail the specific assumptions later), and in particular they are defined and finite at any probability measure with finite second moment. Therefore, they will be assumed to be non-local and regularizing in the measure variable. Furthermore, we are given a time horizon $T>0$ and the intensities of the {Brownian} idiosyncratic and common noises $\beta,\beta_0\in\R$, respectively. Then, the master equation, written for the unknown function $V:(0,T)\times\R^d\times\sP_2(\R^d)\to\R$ reads as 

\small
\begin{equation}\label{eq:master}
\left\{
\begin{array}{rl}
	-\partial_t V(t,x,\mu) + H(x,\mu,\pa_x V) + \m NV(t,x,\mu)\\
	 -\frac{\beta^2}{2} \Delta_{{\rm ind}}V - \frac{\beta_0^2}{2} \Delta_{{\rm com}} V(t,x,\mu) &= 0,\\[3pt]
	&  {\rm{in}}\ (0,T)\times\R^d\times\sP_2(\R^d),\\[3pt] 
	V(T, x, \mu) &= G(x, \mu),\\[3pt]
	& {\rm{in}}\ \R^d\times\sP_2(\R^d),
\end{array}
\right.
\end{equation}
\normalsize
where
\[
\m NV(t,x,\mu) &=  \int_{\R^d} \pa_\mu V(t,x,\mu,\ti x) \cdot \pa_pH(\ti x, \mu, \pa_x V(t, \ti x, \mu)) \dd\mu(\ti x)\\
\Delta_{{\rm ind}}V &= \tr(\pa_{x x} V(t,x,\mu)) + \int_{\R^d} \tr(\pa_{\ti x \mu} V(t,x,\mu,\ti x)) \dd\mu(\ti x)
\]
and 
\[
\Delta_{{\rm com}}V &= \tr(\pa_{x x} V(t,x,\mu)) + \int_{\R^d} \tr(\pa_{\ti x \mu} V(t,x,\mu,\ti x)) \dd\mu(\ti x)\\
& + 2 \int_{\R^d} \tr(\pa_{x \mu} V(t,x,\mu,\ti x)) \dd\mu(\ti x) \\
&+ \int_{\R^d\times\R^d} \tr(\pa_{\mu \mu} V(t,x,\mu,\ti x, \bar x)) \dd\mu(\ti x) \dd\mu(\bar x).
\]
Here $\pa_\mu V$ stands for the so-called {\it Wasserstein gradient} whose definition is given later in the text.

\medskip

The search for well-posedness theories for \eqref{eq:master} has initiated a great program in the theory. In general, this poses great challenges because of the non-local and infinite dimensional character of the PDE. In particular, this PDE does not possess a comparison principle which means that the consideration of viscosity solutions, for instance, would not be feasible in this setting. Therefore, notions of suitable weak solutions could lead to debates, especially if these lack uniqueness principles. However, there is no ambiguity regarding classical solutions. Our focus in this paper will also be on classical solutions, and so, unless otherwise specified, the term {\it well-posedness} should be understood in the sense of classical solutions. Similarly to the theory of finite dimensional conservations laws, when aiming for global classical solutions, it is quite clear that these should be expected only under suitable {\it monotonicity conditions} on the data $H$ and $G$. Such monotonicity conditions are also strongly related to the uniqueness of MFG Nash equilibria. 

\medskip

{\bf Literature review {on the well-posedness of master equations}.} 
To date, there have been different notions of monotonicity conditions proposed on the data $H$ and $G$, which could serve as sufficient conditions for the global well-posedness theory of \eqref{eq:master}. The diversity and richness of these conditions are deeply related to the geometry under the lens of which we look at $\sP_2(\R^d)$. For instance, $\sP_2(\R^d)$ can be seen as a flat convex space, but it is natural to look at it also as a non-negatively curved infinite dimensional manifold, when equipped with suitable metrics. Historically, the so-called {\it Lasry--Lions (LL) monotonicity condition} was the first one, introduced already in the seminal work \cite{LasLio:07}. Geometrically, this is linked to the flat geometry, imposed on $\sP_2(\R^d).$ When it comes to nonlocal Hamiltonians, this notion has been defined and exploited so far only for so-called {\it separable} Hamiltonians, i.e. the ones which have the structure
\begin{align}\label{cond:sep}
H(x,\mu,p):= H_0(x,p) - F(x,\mu),\ \ \forall (x,\mu,p)\in \R^d\times\sP_2(\R^d)\times\R^d, 
\end{align}
for some $H_0$ and $F$. An alternative monotonicity condition is the so-called {\it displacement monotonicity} condition, which does not require the {separable} structural assumption on $H$. This stems from the notion of {\it displacement convexity}, used widely in the context of optimal transport theory. Thus, this is linked to the curved geometry on $\sP_2(\R^d)$. We now give a brief overview of the well-posedness theories for \eqref{eq:master} in these settings and we also mention some alternative, more recently proposed notions of monotonicity conditions.

\medskip

In \cite[Theorem 5.46]{Carmona2018} the authors have shown that the master equation \eqref{eq:master} is globally well-posed if the data are LL monotone and possess additional regularity assumptions. Several other works provide similar conclusions. We refer to \cite[Theorem 2.4.5]{CarDelLasLio} for the case when the physical space is the flat torus instead of $\R^d$ and to \cite[Theorems 56 and 58]{ChaCriDel22} to the case without common noise (i.e. $\beta_0=0$). We refer also to \cite{JakRut} for new results and clarifications regarding the results from \cite{CarDelLasLio}. However, \cite[Theorem 5.46]{Carmona2018} is the closest result for our purposes.

It is also important to mention that all these global well-posedness results in the context of Lasry--Lions monotonicity impose both the separable structure on the Hamiltonian and the presence of a non-degenerate idiosyncratic noise.

In the context of displacement monotonicity global in time well-posedness have been obtained chronologically as follows. \cite{TwoAuthor2022} provided this in the context of deterministic and potential (in particular $\beta=\beta_0=0$ and $H$ separable) games (for similar results, see also \cite{BenGraYam}). \cite{FourAuthor} provided the first global in time well-posedness result in the case of non-separable displacement monotone Hamiltonians and non-degenerate idiosyncratic noise (i.e. $\beta\neq 0$). Finally, \cite{BanMesMou} provided the result in the case of degenerate idiosyncratic noise (i.e. $\beta=0$) and compared to \cite{FourAuthor}, under lower level regularity assumptions on the data, and the weaker version of the displacement monotonicity condition on $H$. 

\medskip

Recently, in \cite{MouZha:2022} and \cite{MouZha:2024} the authors have proposed a {notion of} {\it anti-monotonicity} condition on final data of master equations, which together with other sufficient structural conditions on the Hamiltonian resulted in the the global in time well-posedness of the master equation. We would like to emphasize that for this to hold, the anti-monotonicity condition on the final data has to be carefully chosen in line with the structural conditions on the Hamiltonian. As we show below, this framework can entirely be embedded into our main results under the umbrella of {our newly proposed canonical transformation.} 

Several other recent developments have seen the light in the context of the well-posedness of MFG master equations. For a non-exhaustive list we refer to \cite{AmbMes,Ber,CarCirPor20,CecDel, GraMes:2022, GraMes22b}.

\medskip

{\bf Our contributions.} In this paper our main objective is to explore some geometric features of Hamiltonian systems which could lead to the global well-posedness of the master equation \eqref{eq:master}. The heart of our analysis consist of so-called {\it canonical transformations} which in particular reveal new perspectives on existing and new monotonicity conditions on the Hamiltonians and final data associated to \eqref{eq:master}, and in turn lead to new well-posedness theories. The values of the noise intensities, $\beta, \beta_0$ will not not be significant in our consideration, and our main results hold true also for degenerate problems, i.e. when $\beta=0$ or $\beta_0 = 0$.

\medskip

In classical Hamiltonian mechanics, canonical transformations are coordinate transformations on the phase space, which preserve the structure of Hamilton's equations. 
In symplectic geometry, canonical transforms are known as symplectomorphisms (where the phase space is a cotangent bundle and the symplectic form is the canonical 2-form). Since in our setting we are only concerned with Euclidean space we do not use the symplectic terminology. However, it would be interesting to study how symplectomorphisms could potentially generate new well-posedness theories for Hamilton--Jacobi equations and the master equation in more general settings (i.e. when the underlying space is not Euclidean). We refer the reader to \cite{Arnold1989} for a introduction to applications of symplectic geometry in classical mechanics. We refer also to our companion short note \cite{BanMes:short}, where we explain the regularization effect of such transformations in the case of deterministic finite dimensional HJB equations.

\smallskip

As the master equation has in particular a natural character arising from infinite dimensional Hamiltonian dynamics, we will show below, that such transformations play a deep role in revealing hidden features of it. 

\medskip

Let us describe the driving idea behind our results. For Hamiltonians $H:\R^d\times\sP_2(\R^d)\times\R^d\to\R$ and final data $G:\R^d\times\sP_2(\R^d)\to\R$ we consider a family of prototypical {\it linear} canonical transformations as follows. Let $\alpha\in\R$ and define $H_\alpha: \R^d\times\sP_2(\R^d)\times\R^d\to\R$ and $G_\alpha: \R^d\times\sP_2(\R^d)\to\R$ as
\begin{align}\label{def:H_alpha}
H_\alpha(x,\mu,p) := H(x,\mu,p-\alpha x) \ \ {\rm{and}}\ \ G_\alpha(x,\mu):= G(x,\mu) + \frac{\alpha}{2}|x|^2. 
\end{align}
In particular, this means that the corresponding canonical transformation has the form of 
$$
\R^d\times\sP_2(\R^d)\times\R^d\ni (x,\mu,p)\mapsto (x,\mu, x-\alpha p).
$$
This is a `finite dimensional' transformation, as there is no change in the measure variable $\mu$.
Having defined these transformations, the heart of our analysis is based on the following observation: {\it fix any $\alpha \in \R$, then the master equation with data $(H,G)$ is well-posed if and only if it is well-posed with data $(H_\alpha, G_\alpha)$} (see Theorem \ref{thm:intro1}; in particular the solutions to the corresponding master equations differ only by an explicit function of $(t,x)$, parametrized by $\alpha$).

The message of this {result} is that if one produces a well-posedness theory for the master equation, this will lead to a whole {\it one parameter family} of well-posedness theories, with the transformed data. A deeper consequence of this theorem is the opposite implication. Suppose that one is given the data $(H,G)$. If one is able to find a suitable range of the parameter $\alpha$ such that $(H_\alpha, G_\alpha)$ satisfies some well-known monotonicity conditions, then the problem with the original data must be well-posed. This second one will be the direction that we investigate in this paper.

\smallskip

Fix $\alpha \in \R$. It is easy to see that $G$ is LL monotone, if and only if $G_\alpha$ is LL monotone and the situation is the same for separable $H$. However, as we will show below, this phenomenon is much different in the displacement monotone regime. Therefore the previously described result has powerful applications in the context of displacement monotonicity but not for LL monotonicity.

\medskip

In the main theorem of this paper, Theorem \ref{thm:main2}, we propose easily verifiable sufficient conditions on $H$ to ensure that $H_\alpha$ is displacement monotone. As a consequence, we discover new regimes of global well-posedness of the master equation. In an informal way, this result can be summarized as follows (we refer to Theorem \ref{thm:main2} for the precise statement).
%

\begin{thm}\label{thm:main}
Suppose that $H:\R^d\times\sP_2(\R^d)\times\R^d\to\R$ is twice continuously differentiable with uniformly bounded second order derivatives. Suppose moreover that $H$ is strongly convex in the $p$-variable. 

Suppose that the symmetric part of $\partial_{xp}H$ is bounded below by an explicit quantity depending on the other second derivatives of $H$. Then, $H_{\alpha}$ is displacement monotone for a suitable range of $\alpha\in\R$, depending on the size of the second derivatives of $H$ in a precise way.

Furthermore, if $G:\R^d\times\sP_2(\R^d)\to\R$ is twice continuously differentiable and displacement $\alpha$-monotone for such specific $\alpha$, then the master equation is globally well-posed. 
\end{thm}

\medskip

This theorem has an immediate {implication, coming from a sort of {\it `regularization phenomenon'} of $\partial_{xp}H$. This} can informally be formulated as follows. 

\begin{cor} \label{thm:intro3}
Suppose that $G:\R^d\times\sP_2(\R^d)\to\R$ and $H:\R^d\times\sP_2(\R^d)\times\R^d\to\R$ are twice continuously differentiable with uniformly bounded second order derivatives. Suppose moreover that $H$ is strongly convex in the $p$-variable. 

We have that there exists $C>0$ depending on {second derivatives of $H$ and $G$ (but independent of $T$)} so that if $\alpha \geq C$ then {the master equation is globally well-posed with data} $(\tilde H, G)$, 
where $\tilde H:\R^d\times\sP_2(\R^d)\times\R^d\to\R$ is given by 
$$
\tilde H(x,\mu,p):= H(x,\mu,p)+ \alpha p\cdot x.
$$
\end{cor}

Hence even if we did not know that the original master equation was solvable, the modified master equation is solvable for $\alpha$ large enough. One can compare the Hamiltonian $\ti H$ with the one in \cite[Example 7.2]{MouZha:2024}.

\begin{rmk}
Corollary \ref{thm:intro3} has a deep message: if the Hamiltonian is such that $\partial_{xp}H$ is sufficiently large compared to other second order derivatives of $H$ and the second order derivatives of $G$, then we have a global well-posedness theory for the master equation. Therefore $\partial_{xp}H$, and in particular adding suitable multiples of the function $(x,p,\mu)\mapsto p\cdot x$ to $H$ {can produce a {`regularization effect'} for the master equation, independently of $T>0$}. By carefully examining Lemma \ref{prop:semi-mon}, we see that what is going on is that the $p\cdot x$ term is transformed into a multiple of $\frac{\abs{x}^2}2$, which provides displacement monotonicity for the problem and hence regularizes the master equation. It is easy to see that adding a suitable multiple of the term $\frac{\abs{x}^2}2$ to $H$ produces displacement monotonicity. Clearly, these regularization effects are independent of the noise intensities.
\end{rmk}

\begin{rmk}
We emphasize that the regularization provided by the function $(x,p,\mu)\mapsto\alpha p\cdot x$ in the statement of Corollary \ref{thm:intro3} produces indeed a genuinely new class of data, not covered in the literature before, for which the master equation is globally well-posed. In particular, if we take an arbitrary pair of data $(H,G)$, not satisfying any monotonicity condition (either displacement or LL, if $H$ is separable), it is immediate to check that $\tilde H$ will satisfy neither displacement monotonicity nor LL monotonicity. Therefore, the monotonicity of the pair $(\tilde H, G)$ is indeed hidden. 
\end{rmk}

\medskip

{\bf Further implications of our main results.}  Having our main results in hand, we have revisited some previous well-posedness results from the literature.

\medskip

When $G$ is displacement semi-monotone, then the well-posedness of \eqref{eq:master} can be guaranteed if $H_\alpha$ is displacement monotone for sufficiently large $\alpha$. It turns out that our characterization for this given in Proposition \ref{prop:semi-mon} coincides with the respective assumptions on $H$ discovered recently in \cite{MouZha:2022}. 

\medskip

In the recent paper \cite{MouZha:2024}, the authors proposed a notion of {\it anti-mono\-to\-ni\-ci\-ty} for final data $G$. They have described some sufficient conditions on $H$ and $G$ which result in a global well-posedness theory of \eqref{eq:master}, if $\beta\neq 0$, and  $G$ is suitably anti-monotone. There was an emphasis on the fact that $G$ needed to be {\it `sufficiently'} anti-monotone.  

\medskip

It turns out that these well-posedness results from \cite{MouZha:2024}, under the additional assumptions that $H$ is strictly convex in the $p$-variable fall directly {into the framework} of the canonical transformations and they are an easy consequence of our main results, {in particular Corollary \ref{thm:intro3}}. More precisely, first in Proposition \ref{cor:antiImpliesSemi} we show that if $G$ is $\lambda$-anti-monotone, this implies that it is displacement semi-monotone with a constant which depends {\it only} on $\lambda$ (in particular, the displacement semi-monotonicity constant is independent of the second derivative bounds of $G$). Having strong convexity of $H$ in the $p$-variable, which has also bounded second derivatives allows us to use our Corollary \ref{thm:intro3}. The Hamiltonian considered in \cite{MouZha:2024} has the form of 
$$
H(x,\mu,p):= H_0(x,\mu,p)+ \langle A_0p, x\rangle,
$$
for some constant matrix $A_0\in\R^{d\times d}$. This is slightly different than $\tilde H$ from our Corollary \ref{thm:intro3}, but the term $\langle A_0p, x\rangle$ has exactly the same role as $\alpha p\cdot x$ in our consideration. Therefore, for completeness, as our last contributions, in Proposition \ref{prop:last} and Remark \ref{rmk:last} we show that the assumptions from the main theorem in \cite{MouZha:2024} essentially imply our assumptions. Furthermore, in the case of Hamiltonians which are strongly convex in the $p$-variable, our results need less and weaker assumption, and they hold true without the presence of a non-degenerate idiosyncratic noise. In particular, we demonstrate that the emphasis on the sufficient anti-monotonicity of $G$ in \cite{MouZha:2024} is misleading, and this is not needed. Specifically, in \cite{MouZha:2024} it is remarked: ``\dots we will need to require our data to be sufficiently anti-monotone in appropriate sense''. However we will see that anti-monotonicty is not needed (as anti-monotonicity implies semi-monotonicity) and that \cite{MouZha:2024} has other, more essential assumptions {on $H$} which are what really give the well-posedness result. 

We would like to emphasize that in this paper we provide a general mechanism leading to a global well-posedness theory of master equations, beyond \cite{MouZha:2024}, and the main results from this reference are a consequence of this general theory.

\medskip

{\bf Some concluding remarks.} 

\begin{itemize}

\item For simplicity and transparency of our main ideas, in this manuscript we have decided to focus only on linear canonical transformations of the form $\R^d\times\sP_2(\R^d)\times\R^d\ni (x,\mu,p)\mapsto (x,\mu, x-\alpha p).$ Without much philosophical effort but with significant technical effort, one could consider canonical transformations of the form
$$
\R^d\times\sP_2(\R^d)\times\R^d\ni (x,\mu,p)\mapsto (x,\mu, x-\nabla \varphi(x)),
$$
where $\varphi:\R^d\to\R$ is any given smooth potential function, with bounded second derivatives. In the case of noise, this transformation would lead to the modified Hamiltonians and final data as 
$$ 
H_\varphi(x,\mu,p): = H(x,\mu, x-\nabla \varphi(x)) + \frac{\beta^2+\beta_0^2}{2}\Delta\varphi(x)$$ 
and
$$
G_\varphi(x,\mu):= G(x,\mu) + \varphi(x). 
$$
It is easy to see that Theorem \ref{thm:intro1} holds true if in its  statement $(H_\alpha,G_\alpha)$ is replaced with $(H_\varphi,G_\varphi).$ However, in order to obtain new global well-posedness theory (in the case of potentially degenerate noise), we would need to have a `convexifying regularization' on $G_\varphi$, which means that $\varphi$ would need to be taken to be convex with sufficiently large Hessian eigenvalues. From this point of view, $\varphi(x)=\frac{\alpha}{2}|x|^2$ would be a natural choice, and this is why we have decided to reduce our study to this particular family of potentials.

\smallskip

\noindent We remark that in general Hamiltonians are only defined up to an additive constant. In {classical} mechanics, this is saying that we may pick any value to correspond to the `zero energy'. In the presence of noise the attentive reader will notice that our $H_\alpha$ is not the same as the $H_\varphi$ defined above, {when $\varphi(x)$ is taken to be $\frac{\alpha}{2}|x|^2$} . However, this is not an issue as the difference between the two is a constant. In particular, the two Hamiltonians are equivalent. Thus, we could have defined our $H_\alpha$ as $H_\alpha(x,\mu,p):= H(x,\mu,p-\alpha x) + {\frac{(\beta^2+\beta_0^2)d}{2}}\alpha$ which would then be the exact same as $H_\varphi$ defined above, however this would introduce unnecessary notational clutter.
\item In this paper we have considered only `finite dimensional' canonical transformations (where the measure component stayed fixed). These have proved to have a deep effect on new global well-posedness theories for the master equation. It is a very interesting, but seemingly challenging task to analyze truly infinite dimensional canonical transformations in the context of MFG master equations. In particular it seems that the infinite dimensional canonical transformations do not preserve the structure of MFG, they only preserve the structure of optimal control problems. In this we see a significant difference between games and variational problems.
\end{itemize}

\begin{rmk}
If the Hamiltonian $H$ has an associated Lagrangian with bounded second derivatives we must have that $H$ is strongly convex in $p$. Similarly, the master equation only corresponds to a game, when $H$ is convex in $p$. To the best of the authors knowledge there is no motivation for the master equation outside of this case. 

\smallskip

We remark that if one is interested in the case of non-convex $H$ in $p$ then one can adapt our results by using the Hamiltonian system directly. We refer to the Lagrangian purely for pedagogical reasons and it is not needed for any technical reason. In particular our {canonical} transformation and main theorem, Theorem \ref{thm:intro1}, holds regardless of the convexity of $H$ in $p$.
\end{rmk}

\medskip \medskip

\section{Preliminaries}\label{sec:3}	

In order to keep this discussion self-contained, let us recall some definitions and notations. 

Let $p\ge1$. Based on \cite{AGS}, we recall that the $p$-Wasserstein between $\mu,\nu\in\sP_p(\R^d)$ (probability measures with finite $p$-order moment supported on $\R^d$) is defined as 
$$
W_p^p(\mu,\nu):=\inf\left\{\int_{\R^d\times\R^d}|x-y|^p\dd\gamma(x,y):\ \gamma\in\Pi(\mu,\nu)\right\},
$$  
where $\Pi(\mu,\nu):=\left\{\gamma\in\sP_p(\R^d\times\R^d):\ (p^x)_\sharp\gamma = \mu,\ (p^y)_\sharp\gamma = \nu\right\}$ stands for the set of admissible transport plans in the transportation of $\mu$ onto $\nu$, and $p^x,p^y:\R^d\times\R^d\to\R^d$ denote the canonical projection operators, i.e. $p^x(a,b) = a$ and $p^y(a,b) = b$. We refer to the metric space $(\sP_p(\R^d),W_p)$ as the Wasserstein space.

\medskip

We refer to \cite{AGS,GanTud19} and to \cite[Chapter 5]{CD1} for the notion of Wasserstein differentiability and \emph{fully} $C^k$ functions defined on the Wasserstein space, respectively. Based on \cite{Ahu,CD1,FourAuthor,MesMou} we recall the notion of displacement monotonicity. 

\begin{defin}\label{def:disp_G}
Let $G:\R^d\times \sP_2(\R^d)\to\R$ be a fully $C^1$ function. 
\begin{enumerate}
\item We say that $G$ is \emph{displacement monotone} if 
$$
\int_{\R^d\times\R^d} [\pa_xG(x,\mu)-\pa_x G(y,\nu)]\cdot (x-y)\dd \gamma(x,y)\ge 0,
$$
for any $\gamma\in\Pi(\mu,\nu)$ and for any $\mu,\nu\in\sP_2(\R^d)$. If $G$ is more regular, say fully $C^2$, this definition is equivalent to 
\begin{align*}
&\int_{\R^d}\langle\pa_{xx}G(x,\mu)\xi(x),\xi(x)\rangle\dd\mu(x)\\
& + \int_{\R^d\times\R^d}\langle\pa_{x\mu}G(x,\mu,\tilde x)\xi(x),\xi(\tilde x)\rangle\dd\mu(x)\dd\mu(\tilde x)\ge 0,
\end{align*}
for all $\mu\in\sP_2(\R^d)$ and for all $\xi\in C_c(\R^d;\R^d)$.
\item Based on \cite[Definition 2.7]{FourAuthor}, we say that $G$ is \emph{displacement semi-monotone} or \emph{displacement $\alpha$-monotone}, if there exists $\alpha\in\R$ such that $(x,\mu)\mapsto G(x,\mu)+\frac{\alpha}{2}|x|^2$ is displacement monotone.
\end{enumerate}
\end{defin}

For the corresponding Hamiltonians, we can define the displacement monotonicity condition as follows.
\begin{defin}\label{def:disp_H}
Let $H:\R^d\times\sP_2(\R^d)\times\R^d\to\R$ be such that $H(\cdot,\mu,\cdot)\in C^1(\R^d\times\R^d)$ for all $\mu\in\sP_2(\R^d)$. We say that $H$ is displacement monotone, if 
\begin{align}\label{def:disp_H-1st}
&-\int_{\R^d\times\R^d} [\pa_xH(x,\mu,p^1(x))-\pa_x H(y,\nu,p^2(y))]\cdot (x-y)\dd \gamma(x,y)\\ 
\nonumber&+ \int_{\R^d\times\R^d} [\pa_p H(x,\mu,p^1(x))-\pa_p H(y,\nu,p^2(x))]\cdot (p^1(x)-p^2(y))\dd \gamma(x,y),
\end{align}
for all $\mu,\nu\in\sP_2(\R^d)$, $\gamma\in\Pi(\mu,\nu)$ and for all $p^1,p^2 \in C_b(\R^d;\R^d)$.
\end{defin}

\begin{rmk}
\begin{enumerate}
\item Suppose that $H:\R^d\times\sP_2(\R^d)\times\R^d\to\R$ is fully $C^2$, strictly convex in the $p$-variable and satisfies
\begin{align}\label{def:disp_H-2nd}
&\int_{\R^d\times\R^d}\left[\partial_{x\mu}H(x,\mu,\tilde x,p(x))v(\tilde x)+\partial_{xx}H(x,\mu,p(x))v(x)\right]\cdot v(x)\dd\mu(x)\dd\mu(\tilde x)\\
\nonumber&+\frac{1}{4}\int_{\R^d}\left\{\Big|[\pa_{pp}H(x,\mu,p(x))]^{-\frac12}\int_{\R^d}\pa_{p\mu}H(x,\mu,\tilde x,p(x))v(\tilde x)\dd\mu(\tilde x)\Big|^2\right\}\dd \mu(x)\\
\nonumber&\le 0,
\end{align}
for all $\mu\in\sP_2(\R^d)$, for all $p\in C(\R^d;\R^d)$ and for all $v\in L^2_{\mu}(\R^d;\R^d)$. Then $H$ satisfies the displacement monotonicity condition from Definition \ref{def:disp_H}. For the proof of this fact we refer to \cite[Lemma 2.7]{MesMou}.
\end{enumerate}
\end{rmk}

\begin{defin}{\cite[Definition 3.8]{MouZha:2024},\cite[Definition 3.4]{MouZha:2022}}\label{def:anti} Let $\lambda = (\lambda_0,\lambda_1, \lambda_2,\lambda_3)\in\R^4$ be such that $\lambda_0>0$, $\lambda_1\in\R,$ $\lambda_2>0$ and $\lambda_3\ge 0$. Let $G:\R^d\times\sP_2(\R^d)\to\R$ be fully $C^2$. It is said that $G$ is $\lambda$-anti-monotone, if 
\begin{align*}
&\lambda_0 \int_{\R^d}\langle\pa_{xx}G(x,\mu)\xi(x),\xi(x)\rangle\dd\mu(x)\\ 
& + \lambda_1\int_{\R^d\times\R^d}\langle\pa_{x\mu}G(x,\mu,\tilde x)\xi(x),\xi(\tilde x)\rangle\dd\mu(x)\dd\mu(\tilde x)\\
&+ \int_{\R^d}\abs{\pa_{xx}G(x,\mu)\xi(x)}^2\dd\mu(x)
+ \lambda_2 \int_{\R^d}\Big|\int_{\R^d}\pa_{x\mu}G(x,\mu,\tilde x)\xi(\tilde x)\dd\mu(\tilde x)\Big|^2 \dd \mu(x) \\
& \leq \lambda_3\int_{\R^d} \abs{\xi(x)}^2 \dd \mu(x) 
\end{align*}
for all $\mu\in\sP_2(\R^d)$ and for all $\xi\in L^2_{\mu}(\R^d;\R^d)$.
\end{defin}

\section{New Well-Posedness Theories for MFG and master equations}

We impose a set of assumptions which are going to be imposed for our main results. {These are relatively standard assumptions, which appear naturally in the literature on the well-posedness theories for master equations.}

\begin{assump}\label{hyp:G}
Suppose that $G:\R^d\times\sP_2(\R^d)\to\R$ is fully $C^2$, bounded below and is such that
\begin{itemize}
\item $\pa_{xx}G$ is uniformly continuous and it is uniformly bounded by $L^G$ on $\R^d\times\sP_2(\R^d)$; 
\item $\pa_{x\mu}G$ is uniformly continuous and it is uniformly bounded by $L^G$ on $\R^d\times\sP_2(\R^d)\times\R^d$,
\end{itemize}
for some $L^G>0$.
\end{assump}

\begin{assump}\label{hyp:H}
Suppose that $H:\R^d\times\sP_2(\R^d)\times\R^d\to\R$ is fully $C^2$ and satisfies the followings 
\begin{itemize}
\item $\pa_{pp}H$ is uniformly continuous and $\pa_{pp}H(x,\mu,p) \geq c_0^{-1}I$, for some $c_0>0$ and for all $(x,\mu,p)\in \R^d\times\sP_2(\R^d)\times\R^d$; 
\item $\pa_{xp}H$, $\pa_{pp}H, \pa_{xx}H$ are continuous and are uniformly bounded by $L^H$ on $\R^d\times\sP_2(\R^d)\times\R^d$;
\item $\pa_{p\mu}H, \pa_{x\mu}H$ are uniformly continuous and are uniformly bounded by $L^H$ on $\R^d\times\sP_2(\R^d)\times\R^d\times\R^d$;
\item $\partial_p H(x,\mu,p)\cdot p - H(x,\mu,p)\ge -L^H$ for all $(x,\mu,p)\in \R^d\times\sP_2(\R^d)\times\R^d,$
\end{itemize}
for some $L^H>0$.
\end{assump}

\begin{rmk}
\begin{enumerate}
\item When continuity of functions is assumed in the measure variable, this is with respect to the $W_2$ metric.
\item Assumptions \ref{hyp:G} and \ref{hyp:H} from above are the standing assumptions imposed in \cite{BanMesMou}.
\end{enumerate}
\end{rmk}

{
Let us now restate our crucial observation from the introduction in form of a theorem.
\begin{thm}\label{thm:intro1}
Fix any $\alpha \in \R$. The master equation with data $(H,G)$ is well-posed if and only if it is well-posed with data $(H_\alpha, G_\alpha)$. 
\end{thm}
}
\begin{proof}
Via direct computation we can verify that $V$ is a solution of the master equation with data $(H,G)$ if and only if $\ti V(t,x,\mu) := V(t,x,\mu) +  \frac{\alpha}{2}\abs{x}^2 - \frac{(\beta_0^2+\beta^2)\alpha d}{2}(t-T)$ is a solution of the master equation with data $({H_\alpha, G_\alpha})$. 
\end{proof}

\begin{rmk}
Because of the connection between the solvability of the master equation with data $(H,G)$ and $({H_\alpha, G_\alpha})$ described in Theorem \ref{thm:intro1}, the same connection holds true for the solutions to the corresponding finite dimensional mean field games systems as well.
\end{rmk}

\medskip

Recall the definition \eqref{def:H_alpha}. Now we give some sufficient conditions on Hamiltonians $H$ which would result into the displacement monotonicity of the transformed Hamiltonians $H_\alpha$.

\begin{lem}\label{prop:semi-mon}
Let $H$ be fully $C^2$. Then $H_\alpha$ is displacement monotone if and only if
\small
\begin{align}\label{ineq:semi-monot}
&\int_{\R^d}\left[\Big(\partial_{xx}H(x,\mu,p(x)) - 2\alpha\partial_{xp}H(x,\mu,p(x))\Big)v(x)\right]\cdot v(x)\dd\mu(x)\\ 
\nonumber &+ \int_{\R^d\times\R^d} \left[ \Big(\partial_{x\mu}H(x,\mu,\tilde x,p(x)) - 2\alpha\partial_{p\mu}H(x,\mu,\tilde x,p(x)) \Big)v(\tilde x)\right]\cdot v(x) \dd\mu(x)\dd\mu(\tilde x) \\
\nonumber&+\frac{1}{4}\int_{\R^d}\bigg\{\bigg|[\pa_{pp}H(x,\mu,p(x))]^{-\frac12}\bigg[\int_{\R^d}\pa_{p\mu}H(x,\mu,\tilde x,p(x))v(\tilde x)\dd\mu(\tilde x)\\
\nonumber&+2\alpha\pa_{pp}H(x,\mu,p(x))v(x)\bigg]\bigg|^2\bigg\}\dd \mu(x)\\
\nonumber&\le 0,
\end{align}
\normalsize
for all $\mu\in\sP_2(\R^d)$, for all $p\in C(\R^d;\R^d)$ and for all $v\in L^2_{\mu}(\R^d;\R^d)$.
\end{lem}

\begin{proof}
We readily compute
\begin{align*}
\pa_{xx} \ti H(x,\mu, p) &= \pa_{xx} H(x,\mu,p-\alpha x) - 2\alpha \Real (\pa_{xp} H(x,\mu,p-\alpha x))\\
& + \alpha^2 \pa_{pp} H(x,\mu,p-\alpha x), \\
\pa_{x\mu} \ti H(x,\mu,\cdot, p) &= \pa_{x\mu} H(x,\mu,\cdot,p-\alpha x) - \alpha \pa_{p\mu} H(x,\mu,\cdot,p-\alpha x),\\
\pa_{p\mu} \ti H(x,\mu,\cdot, p) &= \pa_{p\mu} H(x,\mu,\cdot, p-\alpha x),\\
\pa_{pp} \ti H(x,\mu, p) &= \pa_{pp} H(x,\mu, p-\alpha x).
\end{align*}
The result now immediately follows by writing the inequality \eqref{def:disp_H-2nd} for $\ti H$ in terms of $H$, after noting that we may replace $\Real (\pa_{xp} H)$ with $\pa_{xp} H$ since the quadratic form induced by a skew-symmetric operator is null.  
\end{proof}

\begin{rmk}
The inequality in \eqref{ineq:semi-monot} can be equivalently rewritten as 
\begin{align}\label{def:alpha_disp_H}
&\int_{\R^d\times\R^d}\left[\partial_{x\mu}H(x,\mu,\tilde x,p(x))v(\tilde x) -\alpha \partial_{p\mu}H(x,\mu,\tilde x,p(x))v(\tilde x)\right]\cdot v(x)\dd\mu(x)\dd\mu(\tilde x)\\ 
\nonumber&+\int_{\R^d}\big[\partial_{xx}H(x,\mu,p(x))v(x)-2\alpha \partial_{xp}H(x,\mu,p(x))v(x)\\
\nonumber& + \alpha^2\partial_{pp}H(x,\mu,p(x))v(x) \big]\cdot v(x)\dd\mu(x)\\
\nonumber&+\frac{1}{4}\int_{\R^d}\left\{\Big|[\pa_{pp}H(x,\mu,p(x))]^{-\frac12}\int_{\R^d}\pa_{p\mu}H(x,\mu,\tilde x,p(x))v(\tilde x)\dd\mu(\tilde x)\Big|^2\right\}\dd \mu(x)\\
\nonumber &\le 0,
\end{align}
for all $\mu\in\sP_2(\R^d)$, for all $p\in C(\R^d;\R^d)$ and for all $v\in L^2_{\mu}(\R^d;\R^d)$. This is the exact same condition as \cite[(5.10)]{MouZha:2022}. 
\end{rmk}

\medskip

We introduce the following notations. 
$$
\underline{\kappa}(\pa_{xp}H):=\inf_{(x,\mu,p)\in\R^d\times\sP_2(\R^d)\times\R^d}\lambda_{\min}(\Real \pa_{xp}H(x,\mu,p)),
$$
{where for $A\in\R^{d\times d}$, we adopt the notation $\Real (A) := (A+A^\top)/2$ and for $A\in\R^{d\times d}$ symmetric $\lambda_{\min}(A)$ stands for its smallest eigenvalue.}
Furthermore, to denote the suprema of the standard $2$-matrix norms, we use the notation
\begin{align*}
&\abs{\pa_{x\mu} H}:= \sup_{(x,\mu,p,\tilde x)\in\R^d\times\sP_2(\R^d)\times\R^d\times\R^d}\abs{\pa_{x\mu} H(x,\mu,p,\tilde x)};\\ 
&\abs{\pa_{p\mu} H}:=  \sup_{(x,\mu,p,\tilde x)\in\R^d\times\sP_2(\R^d)\times\R^d\times\R^d}\abs{\pa_{p\mu} H(x,\mu,p,\tilde x)};\\ 
&\abs{\pa_{xx} H}:=\sup_{(x,\mu,p)\in\R^d\times\sP_2(\R^d)\times\R^d}\abs{\pa_{xx} H(x,\mu,p)},
\end{align*}
and so on for similar quantities. Now, we can formulate the {second main result of our paper}. 
\begin{thm}\label{thm:main2}
Suppose that $H:\R^d\times\sP_2(\R^d)\times\R^d\to\R$ satisfies 
$$\pa_{pp}H(x,\mu,p) \geq c_0^{-1}I,$$ for some $c_0>0$ and for all $(x,\mu,p)\in \R^d\times\sP_2(\R^d)\times\R^d$. Suppose that $\underline{\kappa}(\pa_{xp}H), \abs{\pa_{pp} H}, \abs{\pa_{xx} H}, \abs{\pa_{p\mu} H}$ and $\abs{\pa_{x\mu} H}$ are finite. Define 
$$L_{our}^{H}:=\abs{\pa_{x\mu} H}  + \frac14 c_0\abs{\pa_{p\mu} H}^2+ \abs{\pa_{xx} H}.$$ 
Suppose that $\underline{\kappa}(\pa_{xp} H) \geq \frac12\abs{\pa_{p\mu}H} +  \sqrt{\abs{\pa_{pp} H} L_{our}^{H}}$. Then $H_\alpha$ is displacement monotone for any 
$$\alpha\in \left[\alpha^H_-,\alpha^H_+\right],$$
where
$$
\alpha^H_\pm:= \frac{\underline{\kappa}(\pa_{xp} H) - \frac12\abs{\pa_{p\mu}H}\pm \sqrt{\left(\underline\kappa(\pa_{xp} H) - \frac12\abs{\pa_{p\mu}H}\right)^2-\abs{\pa_{pp} H} L_{our}^{H}}}{\abs{\pa_{pp} H}}.
$$
In particular we have the result for $\alpha := \frac{\underline{\kappa}(\pa_{xp} H) - \frac12\abs{\pa_{p\mu}H}}{\abs{\pa_{pp} H}}$.
\end{thm}
\begin{proof}
For $\alpha \in \left[\alpha^H_-,\alpha^H_+\right]$, $\mu\in\sP_2(\R^d)$, $p\in C(\R^d;\R^d)$ and for $v\in L^2_{\mu}(\R^d;\R^d)$ normalized, i.e. $\int_{\R^d}|v(x)|^2 d\mu=1$, we compute
\[
&\int_{\R^d\times\R^d}\left[\partial_{x\mu}H(x,\mu,\tilde x,p(x))v(\tilde x) -\alpha \partial_{p\mu}H(x,\mu,\tilde x,p(x))v(\tilde x)\right]\cdot v(x)\dd\mu(x)\dd\mu(\tilde x) \\
& +\int_{\R^d}\big[\partial_{xx}H(x,\mu,p(x))v(x) -2\alpha \partial_{xp}H(x,\mu,p(x))v(x)\\
& + \alpha^2\partial_{pp}H(x,\mu,p(x))v(x) \big]\cdot v(x)\dd\mu(x)\\
\nonumber&+\frac{1}{4}\int_{\R^d}\left\{\Big|[\pa_{pp}H(x,\mu,p(x))]^{-\frac12}\int_{\R^d}\pa_{p\mu}H(x,\mu,\tilde x,p(x))v(\tilde x)\dd\mu(\tilde x)\Big|^2\right\}\dd \mu(x) \\
&\leq \int_{\R^d\times\R^d}[\abs{\partial_{x\mu}H} +\alpha \abs{\partial_{p\mu}H} +\abs{\partial_{xx}H} 
-2\alpha \underline{\kappa}(\partial_{xp}H) + \alpha^2\abs{\partial_{pp}H} ]\dd\mu(x)\dd\mu(\tilde x)\\
\nonumber&+\frac{c_0}{4}\int_{\R^d}\left\{\Big|\int_{\R^d}\abs{\partial_{p\mu}H}\dd\mu(\tilde x)\Big|^2\right\}\dd \mu(x) \\
&= \abs{\partial_{xx}H}
-2\alpha \underline{\kappa}(\partial_{xp}H) + \alpha^2\abs{\partial_{pp}H} +
\abs{\partial_{x\mu}H} +\alpha \abs{\partial_{p\mu}H}+\frac{c_0\abs{\partial_{p\mu}H}^2}{4} \\
&= \abs{\partial_{pp}H}\alpha^2 -2\left(\underline{\kappa}(\partial_{xp}H) - \frac12\abs{\pa_{p\mu}H}\right)\alpha  +  \abs{\partial_{xx}H} + \abs{\partial_{x\mu}H}+\frac{c_0\abs{\partial_{p\mu}H}^2}{4} \\
&= \abs{\partial_{pp}H}\alpha^2 -2 \left(\underline{\kappa}(\partial_{xp}H) - \frac12\abs{\pa_{p\mu}H}\right)\alpha + L_{our}^{H} \\
& \le 0,
\]
where in the last inequality we used the sign of the quadratic expression.
\end{proof}
As an immediate consequence of Theorem \ref{thm:main2}, we have the well-posedness result {in Corollary \ref{thm:intro3}}.

\begin{proof}[Proof of {Corollary} \ref{thm:intro3}]
We see that all second order derivatives of $\tilde H$ and $H$ match, except the ones involving $\pa_{xp}$, for which we have
$$
\pa_{xp}\tilde H = \pa_{xp}H +\alpha I.
$$
By the uniform bounds on the corresponding second order derivatives of $H$, we see that for $\alpha$ sufficiently large, $\tilde H$ fulfills the assumptions of Theorem \ref{thm:main2}. Increasing $\alpha$ further if necessary, we can ensure that $G$ is displacement $\alpha$-monotone. Having $G$ displacement $\alpha$-monotone and $H_\alpha$ displacement monotone would result via Theorem \ref{thm:intro1} in the desired global well-posedness result for the master equation.
\end{proof}

\subsection{Our results and previous results on the master equation involving displacement {\emph{semi-monotone}} data}\label{subsec:semi-mon}

We notice that the inequality \eqref{ineq:semi-monot} is precisely the inequality (5.10) from \cite{MouZha:2022}. This means in particular that \cite[Theorem 5.6]{MouZha:2022} is a direct consequence of Theorem \ref{thm:intro1} and Remark \ref{prop:semi-mon} above.

We note that Theorem \ref{thm:intro1} shows that we have a global well-posedness theory for the master equation as long as $G$ is displacement semi-monotone and the corresponding $\ti H$ is displacement monotone. In particular, it is enough for these to satisfy the `first order' monotonicity conditions, in the sense of Definition \ref{def:disp_G}(1) and \eqref{def:disp_H-1st}. Therefore, Theorem \ref{thm:intro1} together with the well-posedness results from \cite{BanMesMou} provide a more general result than the one in \cite[Theorem 5.6]{MouZha:2022}.

\subsection{Our results and previous results on the master equation involving {\emph{anti-monotone}} data}
Our first objective in this subsection is to show that any function $G:\R^d\times\sP_2(\R^d)\to\R$ which is $\lambda$-anti-monotone in the sense of Definition \ref{def:anti} is actually displacement $\alpha$-monotone in the sense of Definition \ref{def:disp_G}(2), where $\alpha$ can be computed  explicitly  in terms of $\lambda=(\lambda_0,\lambda_1, \lambda_2,\lambda_3)$. We start with some preparatory results.

\begin{rmk}\label{prop:anti_G}
$G$ is $\lambda$-anti-monotone in the sense of Definition \ref{def:anti} with $\lambda=(\lambda_0,\lambda_1, \lambda_2,\lambda_3)$ if and only if
\begin{align*}
&\int_{\R^d} \bigg\{\abs{\pa_{xx}G(x,\mu)\xi(x) + \frac {\lambda_0} 2 \xi(x)}^2\\ 
&+ \lambda_2\abs{ \int_{\R^d} \pa_{x\mu}G(x,\mu,\tilde x)\xi(\tilde x)\dd\mu(\tilde x) + \frac {\lambda_1} {2\lambda_2} \xi(x)}^2\bigg\} \dd\mu(x) \\ 
\nonumber&\leq \(\lambda_3 + \(\frac {\lambda_0} 2\)^2 + \lambda_2\(\frac {\lambda_1} {2\lambda_2}\)^2\)\int_{\R^d} \abs{\xi(x)}^2 \dd \mu(x).
\end{align*}
\end{rmk}
\begin{proof}
This is immediate by an algebraic manipulation after computing the squares.
\end{proof}

\begin{prop}\label{cor:antiImpliesSemi}
If $G$ is $\lambda$-anti monotone in the sense of Definition \ref{def:anti} with $\lambda=(\lambda_0,\lambda_1, \lambda_2,\lambda_3)$, then 
\[
&\abs{\int_{\R^d\times\R^d}\langle\pa_{x\mu}G(x,\mu,\tilde x)\xi(x),\xi(\tilde x)\rangle\dd\mu(x)\dd\mu(\tilde x)}\\    
&\qquad \qquad\leq \(\frac {\abs{\lambda_1}} {2\lambda_2}  + \sqrt{\frac{\lambda_3}{\lambda_2} + \frac{{\lambda_0} ^2}{4\lambda_2} + {\(\frac {\lambda_1} {2\lambda_2}\)^2}} \)\int_{\R^d} \abs{\xi(x)}^2 \dd \mu(x)
\]
and 
\[
&\abs{\int_{\R^d}  \inn{\pa_{xx}G(x,\mu)\xi(x)} { \xi(x)}\dd\mu(x)}\\   
&\qquad \qquad \leq \(\frac{\abs{\lambda_0}}2 + 
\sqrt{\lambda_3 + \(\frac {\lambda_0} 2\)^2 + \lambda_2\(\frac {\lambda_1} {2\lambda_2}\)^2}  \) \int_{\R^d} \dd\mu(x) \abs{\xi(x)}^2. 
\]
In particular $G$ is displacement $\alpha_\lambda$-monotone, with 
\small
$$
\alpha_\lambda\ge\max\left\{\frac {\abs{\lambda_1}} {2\lambda_2}  + \sqrt{\frac{\lambda_3}{\lambda_2} + \frac{{\lambda_0} ^2}{4\lambda_2} + {\(\frac {\lambda_1} {2\lambda_2}\)^2}}; \frac{\abs{\lambda_0}}2 + 
\sqrt{\lambda_3 + \(\frac {\lambda_0} 2\)^2 + \lambda_2\(\frac {\lambda_1} {2\lambda_2}\)^2}\right\}.
$$ 
\normalsize
\end{prop}

\begin{proof}
Let us recall that in the definition of $\lambda$-anti-monotonicity we have $\lambda_0>0$, $\lambda_2>0$, $\lambda_3\ge 0$ and there is no sign restriction on $\lambda_1$.

First, let us suppose that $\lambda_1 \neq 0$. 

Note that for any $v, w\in\R^d$ and any $C > 0$ we have
\[
\abs{\inn vw} \leq \frac{C+2}2 \abs{v}^2 + \frac{1}{2C}\abs{v+w}^2. 
\] 
With the choice of $v:=\frac {\lambda_1} {2\lambda_2} \xi(x)$ and $w:=\int_{\R^d} \pa_{x\mu}G(x,\mu,\tilde x)\xi(\tilde x)\dd\mu(\tilde x)$, we obtain 
\[
&\int_{\R^d}  \abs{\left\langle\int_{\R^d} \pa_{x\mu}G(x,\mu,\tilde x)\xi(\tilde x)\dd\mu(\tilde x),\frac {\lambda_1} {2\lambda_2} \xi(x)\right\rangle}\dd\mu(x)  \\
&\leq \int_{\R^d} \bigg\{ \(\frac {C}2 + 1\) \abs{\frac {\lambda_1} {2\lambda_2} \xi(x)}^2\\
& + \frac{1}{2C}\abs{\int_{\R^d}  \pa_{x\mu}G(x,\mu,\tilde x)\xi(\tilde x)\dd\mu(\tilde x) + \frac {\lambda_1} {2\lambda_2} \xi(x)}^2\bigg\}\dd\mu(x) \\
&= \int_{\R^d} \bigg\{ \(\frac {C}2 + 1\) \abs{\frac {\lambda_1} {2\lambda_2} \xi(x)}^2\\
& + \frac{1}{2C\lambda_2} \(\lambda_2 \abs{\int_{\R^d}  \pa_{x\mu}G(x,\mu,\tilde x)\xi(\tilde x)\dd\mu(\tilde x) + \frac {\lambda_1} {2\lambda_2} \xi(x)}^2 \)\bigg\}\dd\mu(x) \\
&\leq \int_{\R^d} \bigg\{ \(\frac {C}2 + 1\) \abs{\frac {\lambda_1} {2\lambda_2} \xi(x)}^2\\ 
&+ \frac{1}{2C \lambda_2} \(\lambda_3 + \(\frac {\lambda_0} 2\)^2 + \lambda_2\(\frac {\lambda_1} {2\lambda_2}\)^2\) \abs{\xi(x)}^2 \bigg\}\dd\mu(x)
\]
where the last inequality follows from Proposition \ref{prop:anti_G}. Hence,
\small
\[
&\int_{\R^d} \abs{\left\langle\int_{\R^d} \pa_{x\mu}G(x,\mu,\tilde x)\xi(\tilde x)\dd\mu(\tilde x),\xi(x)\right\rangle} \dd\mu(x)  \\
&\leq \( \(\frac {C}2 + 1\) \frac {\abs{\lambda_1}} {2\lambda_2} + \frac{1}{C {\abs{\lambda_1}}} \(\lambda_3 + \(\frac {\lambda_0} 2\)^2 + \lambda_2\(\frac {\lambda_1} {2\lambda_2}\)^2\) \)\int_{\R^d} \abs{\xi(x)}^2 \dd \mu(x) \\
&= \( \frac {\abs{\lambda_1}} {2\lambda_2}  + \frac {C{\abs{\lambda_1}}} {4\lambda_2} + \frac{1}{C {\abs{\lambda_1}}} \(\lambda_3 + \(\frac {\lambda_0} 2\)^2 + \lambda_2\(\frac {\lambda_1} {2\lambda_2}\)^2\) \)\int_{\R^d} \abs{\xi(x)}^2 \dd \mu(x).    
\]
\normalsize
We now take $C = \frac{1}{\abs{\lambda_1}}\sqrt{\({\lambda_3 + \(\frac {\lambda_0} 2\)^2 + \lambda_2\(\frac {\lambda_1} {2\lambda_2}\)^2}\){(4\lambda_2)}}$ to obtain 
\[
&\int_{\R^d} \abs{\left\langle\int_{\R^d} \pa_{x\mu}G(x,\mu,\tilde x)\xi(\tilde x)\dd\mu(\tilde x),\xi(x)\right\rangle} \dd\mu(x)  \\
&\leq \(\frac {\abs{\lambda_1}} {2\lambda_2}  +2 \sqrt{\frac{\lambda_3 + \(\frac {\lambda_0} 2\)^2 + \lambda_2\(\frac {\lambda_1} {2\lambda_2}\)^2}{4\lambda_2}} \)\int_{\R^d} \abs{\xi(x)}^2 \dd \mu(x) \\
&= \(\frac {\abs{\lambda_1}} {2\lambda_2}  + \sqrt{\frac{\lambda_3}{\lambda_2} + \frac{{\lambda_0} ^2}{4\lambda_2} + {\(\frac {\lambda_1} {2\lambda_2}\)^2}} \)\int_{\R^d} \abs{\xi(x)}^2 \dd \mu(x)
\]
{Now, as the left hand side of this estimate is continuous at $\lambda_1=0$, we can send $\lambda_1\to 0$, and conclude the claim for general $\lambda_1\in\R.$}

In the same manner with the choice of $v:=\frac{\lambda_0}2 \xi(x)$ and $w:=\pa_{xx}G(x,\mu)\xi(x)$, for $C>0$ arbitrary we get
\small
\[
&\int_{\R^d}  \abs{\inn{\pa_{xx}G(x,\mu)\xi(x)} { \xi(x)}}\dd\mu(x) \\
&\leq \frac 2{\abs{\lambda_0}}  \int_{\R^d} \(\frac{C+2}2 \abs{\frac{\lambda_0}2 \xi(x)}^2 + \frac{1}{2C} \abs{\pa_{xx}G(x,\mu)\xi(x) + \frac {\lambda_0} 2 \xi(x)}^2  \)\dd\mu(x) \\
&\leq \frac 2{\abs{\lambda_0}}  \(\frac{C+2}2 \(\frac{\lambda_0^2}4\)  + \frac{1}{2C} \(\lambda_3 + \(\frac {\lambda_0} 2\)^2 + \lambda_2\(\frac {\lambda_1} {2\lambda_2}\)^2\)  \) \int_{\R^d} \dd\mu(x) \abs{\xi(x)}^2 \\
&=  \(\frac{\abs{\lambda_0}}2 + \frac{C \abs{\lambda_0}}4 + \frac{1}{\abs{\lambda_0}C} \(\lambda_3 + \(\frac {\lambda_0} 2\)^2 + \lambda_2\(\frac {\lambda_1} {2\lambda_2}\)^2\)  \) \int_{\R^d} \dd\mu(x) \abs{\xi(x)}^2 
\]
\normalsize
By taking $C = \frac2{\abs{\lambda_0}}\sqrt{\({\lambda_3 + \(\frac {\lambda_0} 2\)^2 + \lambda_2\(\frac {\lambda_1} {2\lambda_2}\)^2}\)}$ we obtain
the result. 
\end{proof}
\begin{rmk}
In Proposition \ref{cor:antiImpliesSemi} we see that the estimates, and hence the conclusion regarding the displacement $\alpha$-monotonicity, hold true even for $\lambda_0\le0$. Therefore, we might drop the requirement $\lambda_0>0$, and our claims from below will remain true.
\end{rmk}

\begin{cor}\label{cor:anti-semi}
Let $G:\R^d\times\sP_2(\R^d)\to\R$ be $\lambda$-anti-monotone which satisfies Assumption \ref{hyp:G}. Suppose that $H:\R^d\times\sP_2(\R^d)\times\R^d\to\R$ satisfies Assumption \ref{hyp:H} and it is such that $H_{\alpha_\lambda}$ is displacement monotone, where the constant $\alpha_\lambda$ is given in Proposition \ref{cor:antiImpliesSemi}. Then, the master equation \eqref{eq:master} with data $(H,G)$ is globally well-posed.
\end{cor}

\begin{proof}
This is a direct consequence of Proposition \ref{cor:antiImpliesSemi} and Theorem \ref{thm:intro1}.
\end{proof}

We would like to conclude our paper by showing that, if $H$ is strictly convex in the $p$-variable, then the main theorem on the global well-posedness of the master equation from \cite[Theorem 7.1]{MouZha:2024} 
is a particular case of our main results {from Corollary \ref{cor:anti-semi}}. For completeness, we informally state this here. 

\begin{thm}{\cite[Theorem 7.1]{MouZha:2024}}\label{thm:anti-mon_chenchen} Suppose that $G:\R^d\times\sP_2(\R^d)$ is smooth enough with uniformly bounded second, third and fourth order derivatives. Suppose that the Hamiltonian $H:\R^d\times\sP_2(\R^d)\times\R^d\to\R$ has the specific factorization 
\begin{align*}
H(x,\mu,p):=\langle A_0x,p\rangle + H_0(x,\mu,p),
\end{align*}
for a constant matrix $A_0\in\R^{d\times d}$ and $H_0:\R^d\times\sP_2(\R^d)\times\R^d\to\R$ smooth enough. Suppose furthermore that $G$ is $\lambda$-anti-monotone and that a special set of specific assumption take place jointly for $\lambda=(\lambda_0,\lambda_1, \lambda_2,\lambda_3)$, the matrix $A_0$ and $H_0$. Then the master equation \eqref{eq:master} is globally well-posed for any $T>0$, in the classical sense.
\end{thm}

\medskip

\begin{prop}\label{prop:last}
Suppose that $G:\R^d\times\sP_2(\R^d)\to\R$ is $\lambda$-anti monotone and satisfies Assumption \ref{hyp:G}. Suppose that $H:\R^d\times\sP_2(\R^d)\times\R^d\to\R$ is given by
$$H(x,\mu,p) = \inn{A_0 x}p + H_0(x,\mu,p),$$ 
with $H_0:\R^d\times\sP_2(\R^d)\times\R^d\to\R$ {satisfying Assumption \ref{hyp:H} and $A_0\in\R^{d\times d}$ is a given constant matrix}. 
{Let $K_H := c_0\abs{\pa_{pp} H} = c_0\abs{\pa_{pp} H_0}$ be the condition number of $\pa_{pp} H$.}   Suppose that 
\begin{align}\label{ineq:ours}
\underline{\kappa}(A_0) \geq\max\left\{ \(\frac72+\frac{\sqrt{K_H}}{2}\) L_2^{H_0} + \sqrt{\abs{\pa_{pp}H} \abs{\pa_{xx}H_0} }; \({\frac32} + f(\lambda)\)L_2^{H_0}\right\},
\end{align}
where $\lambda = (\lambda_0,\lambda_1,\lambda_2,\lambda_2)$, we have set
\begin{align*}
f(\lambda) &:= \frac {5{\abs{\lambda_1}}} {4\lambda_2}  + 1 + {\frac{\lambda_3}{2\lambda_2}} + \frac{\lambda_0}{4\lambda_2} + \frac{5\lambda_0}{4}+
\frac{\lambda_3}2 + \frac{{\abs{\lambda_1}}}{4}\\ 
& = 1 + \frac12\left(\frac{5\lambda_0}{2}+  \frac{{\abs{\lambda_1}}}{2} + 
\lambda_3\right) + \frac{1}{2\lambda_2}\left(\frac{\lambda_0}{2} + \frac {5{\abs{\lambda_1}}} {2} + \lambda_3\right),
\end{align*}
{and $L_2^{H_0}>0$ is a constant associated to $H_0$, satisfying
$$|\pa_{xp}H_0|\le L_2^{H_0},\ |\pa_{pp}H_0|\le L_2^{H_0},\ |\pa_{x\mu}H_0|\le L_2^{H_0}\ {\rm{and}}\ |\pa_{p\mu}H_0|\le L_2^{H_0}.$$
}

Then the master equation is globally well-posed. 
\end{prop}
\begin{proof}
{Let us note that by the definition of $L_2^{H_0}$ and by the definition of $L^{H_0}_{our}$, we have that 
\begin{align}\label{ineq:mz}
L^{H_0}_{our}\le L_2^{H_0} + \frac{c_0}{4}\(L^{H_0}\)^2 + |\pa_{xx}H_0|.
\end{align}
}

{As $\underline{\kappa}(\pa_{xp}H) \geq \underline{\kappa}(A_0) - \abs{\pa_{xp} H_0}$}, we see that  the assumption $\underline{\kappa}(A_0) \geq (\frac72+\frac{\sqrt{K_H}}{2}) L_2^{H_0} + \sqrt{\abs{\pa_{pp}H} \abs{\pa_{xx}H_0} }$ {and \eqref{ineq:mz}} imply
\small
\begin{align*}
\underline{\kappa}({\pa_{xp}H})& \geq  \underline{\kappa}(A_0) - \abs{\pa_{xp} H_0} \ge 3 L_2^{H_0} + \frac12\abs{\pa_{p\mu}H} - \abs{\pa_{xp} H_0} + \frac{\sqrt{K_H}}{2} L_2^{H_0}\\
& + \sqrt{\abs{\pa_{pp}H} \abs{\pa_{xx}H_0} } \\
&\ge  2 L_2^{H_0} + \frac12\abs{\pa_{p\mu}H} + \frac{\sqrt{K_H}}{2} L_2^{H_0} + \sqrt{\abs{\pa_{pp}H} \abs{\pa_{xx}H_0} }\\
& =  \frac12\abs{\pa_{p\mu}H} + \sqrt{\frac{c_0}{4}|\pa_{pp}H_0|\(L_2^{H_0}\)^2}+ \sqrt{4\left( L_2^{H_0}\right)^2} + \sqrt{\abs{\pa_{pp}H} \abs{\pa_{xx}H_0} }\\
&\ge \frac12\abs{\pa_{p\mu}H} + \sqrt{\frac{c_0}{4}|\pa_{pp}H_0|\(L_2^{H_0}\)^2}+ \sqrt{4|\pa_{pp}H_0| L_2^{H_0}} + \sqrt{\abs{\pa_{pp}H} \abs{\pa_{xx}H_0} }\\ 
&\geq \frac12\abs{\pa_{p\mu}H_0}  + \sqrt{\abs{\pa_{pp} H_0} L_{our}^{H_0}}\\
&= \frac12\abs{\pa_{p\mu}H}  + \sqrt{\abs{\pa_{pp} H} L_{our}^{H}}
\end{align*}
\normalsize
and so we can apply Theorem \ref{thm:main2}. We get that $H$ is displacement $\alpha$-monotone with 
\begin{align*}
\alpha &= \frac{\underline{\kappa}(\pa_{xp} {H}) - \frac12 \abs{\pa_{p\mu}H}}{\abs{\pa_{pp} H}}\\
&\ge { \frac{\underline{\kappa}(A_0) - |\pa_{xp}H_0| - \frac12 \abs{\pa_{p\mu}H_0}}{\abs{\pa_{pp} H_0}}}\\
&\ge { \frac{ \(\frac32 + f(\lambda)\)L_2^{H_0}- |\pa_{xp}H_0| - \frac12 \abs{\pa_{p\mu}H_0}}{\abs{\pa_{pp} H_0}}}\\
& \geq f(\lambda).
\end{align*}
From Proposition \ref{cor:antiImpliesSemi} we see that $G$ is semi-monotone with constant 
\small
\[
\eta 
&:= \frac {{\abs{\lambda_1}}} {2\lambda_2}  + \sqrt{\frac{\lambda_3}{\lambda_2} + \frac{{\lambda_0} ^2}{4\lambda_2} + {\(\frac {\lambda_1} {2\lambda_2}\)^2}} +
\frac{\lambda_0}2 + 
\sqrt{\lambda_3 + \(\frac {\lambda_0} 2\)^2 + \lambda_2\(\frac {\lambda_1} {2\lambda_2}\)^2} \\
&\leq \frac {{\abs{\lambda_1}}} {2\lambda_2}  + \sqrt{\frac{\lambda_3}{\lambda_2}} + \sqrt{\frac{{\lambda_0} ^2}{4\lambda_2}} + \sqrt{{\(\frac {\lambda_1} {2\lambda_2}\)^2}} +
\frac{\lambda_0}2 + 
\sqrt{\lambda_3} + \sqrt{\(\frac {\lambda_0} 2\)^2} + \sqrt{\lambda_2\(\frac {\lambda_1} {2\lambda_2}\)^2} \\
&\leq \frac {{\abs{\lambda_1}}} {\lambda_2}  + \sqrt{\frac{\lambda_3}{\lambda_2}} + \sqrt{\frac{{\lambda_0} ^2}{4\lambda_2}} +
{\lambda_0} + 
\sqrt{\lambda_3} + \sqrt{\frac {\lambda_1^2} {4\lambda_2}} \\
&\leq \frac {{\abs{\lambda_1}}} {\lambda_2}  + \frac12 + {\frac{\lambda_3}{2\lambda_2}} + \frac{\lambda_0}{4\lambda_2} + \frac{\lambda_0}{4}+
{\lambda_0} + 
\frac12 + \frac{\lambda_3}2 + {\frac {{\abs{\lambda_1}}} {4\lambda_2}} + \frac{{\abs{\lambda_1}}}{4}\\
&= \frac {5{\abs{\lambda_1}}} {4\lambda_2}  + 1 + {\frac{\lambda_3}{2\lambda_2}} + \frac{\lambda_0}{4\lambda_2} + \frac{5\lambda_0}{4}+
\frac{\lambda_3}2 + \frac{{\abs{\lambda_1}}}{4} \\
&= f(\lambda)
\]
\normalsize
and so the result follows. 
\end{proof}

\begin{rmk}\label{rmk:last}
We compare Proposition \ref{prop:last} with \cite[Theorem 7.1]{MouZha:2024}. This theorem has many assumptions. We show that up to constants (depending only on $K_H$) only a few of these many assumptions imply our assumptions. {First, we recall that the definition of the $3\times 3$ matrices $A_1,A_2$ from formula \cite[(4.3)]{MouZha:2024}. These are not constructed from $A_0$ above, and they involve constants coming in particular from $\lambda=(\lambda_0,\lambda_1,\lambda_2,\lambda_3).$ Furthermore, for $A\in\R^{d\times d}$,  $\bar\kappa(A)$ stands for the largest eigenvalue of $\Real(A).$}

To continue we need the assumption
\begin{align}\label{ineq:theirs}
\underline{\kappa}(A_0) \geq (1+\bar{\kappa}(A_1^{-1}A_2)) L_2^{H_0}. 
\end{align}

In \cite[Theorem 7.1]{MouZha:2024} (specifically the second item of (7.1)) it is assumed that
\begin{align}\tag{\ref{ineq:theirs}$'$}\label{ineq:theirsPrime}
\underline{\kappa}(A_0) \geq (1+\underline{\kappa}(A_1^{-1}A_2)) L_2^{H_0}, 
\end{align}
although they probably meant to assume \eqref{ineq:theirs}\footnote{   
	The $\underline{\kappa}$ on the right-hand side is likely a typo as in the fourth to last line on \cite[page 15]{MouZha:2024} the authors need to use $\bar{\kappa}(A_1^{-1}A_2)$. Furthermore we see $\bar{\kappa}$ appearing correctly also {in a similar assumption, \cite[(6.3)]{MouZha:2022}}.} .

We can formulate the following statement. 
\medskip

{\bf Claim.} {\it The assumptions of \cite[Theorem 7.1]{MouZha:2024}, up to a multiplicative constant depending on $K_H$, imply \eqref{ineq:ours}. }

\medskip

{\bf Proof of claim.} By definition, we have that $\bar{\kappa}(A_1^{-1}A_2) \geq v^\top A_1^{-1}A_2 v$ for any unit vector $v\in\R^3$. Taking $v = \frac{1}{\sqrt 3}(1,1,1)^\top$ and using the explicit form of $A_1, A_2$ given in \cite[(4.3)]{MouZha:2024} {together with the fact that all the entries of these matrices are non-negative,} by direct computation we obtain
\footnotesize
\[
\bar{\kappa}(A_1^{-1}&A_2) 
\geq \frac{1}{3} \( \frac14 \( \lambda_0 + \lambda_0 +   \abs{\lambda_0 - \frac12 \lambda_1} + \lambda_3 \) + \frac1{{2}\lambda_2} \( \lambda_0 + \abs{\lambda_1} + (\frac12 \abs{\lambda_1} + \lambda_2 + \lambda_3) \) \) \\
&\geq \frac{1}{3} \( \frac14 \( \lambda_0 + \lambda_0 -   \abs{\lambda_0} + \frac12 \abs{\lambda_1} + \lambda_3 \) + \frac1{{2}\lambda_2} \( \lambda_0 + \abs{\lambda_1} + (\frac12 \abs{\lambda_1} + \lambda_2 + \lambda_3) \) \) \\
&= \frac{1}{3} \( \frac14 \( \lambda_0 + \frac12 \abs{\lambda_1} + \lambda_3 \) + \frac1{{2}\lambda_2} \( \lambda_0 + \abs{\lambda_1} + (\frac12 \abs{\lambda_1} + \lambda_2 + \lambda_3) \) \) \\
&\geq \frac{1}{{15}} f(\lambda),
\]
\normalsize

\medskip

so \eqref{ineq:theirs} 
implies that 
\begin{align}\label{ineq:last11}
\underline{\kappa}(A_0) \geq \frac{1}{15} L_2^{H_0}\({15} + f(\lambda)\).
\end{align} 

Furthermore we see from the second inequality in \cite[(7.2)]{MouZha:2024} that 
$$\bar{\gamma}\underline{\kappa}(A_0) \geq \abs{\pa_{xx}H}.$$ By the assumption (i) of \cite[Theorem 7.1]{MouZha:2024} we have that $\bar{\gamma}$ satisfies \cite[(4.2)]{MouZha:2024} in which the first inequality implies that $\lambda_0 > \frac{\bar{\gamma}^2}{4 \underline{\gamma}} - \frac{8\lambda_3}{4\underline{\gamma}}$. Hence we obtain $(4 \underline{\gamma} \lambda_0 + 8 \lambda_3) \geq \bar{\gamma}^2$. It is clear that $2f(\lambda) \geq \lambda_0$ and $2f(\lambda) \geq \lambda_3$, therefore we get $16f(\lambda)(1 + \underline{\gamma}) \geq \bar{\gamma}^2$. Since $\underline{\gamma} < \bar{\gamma}$ by assumption (i) of \cite[Theorem 7.1]{MouZha:2024} and $1 < \bar{\gamma}$ by the same assumption we get $2\bar{\gamma} \geq 1 + \underline{\gamma}$ and so we obtain $32f(\lambda) \geq \bar{\gamma}$. 
Hence we get 
\[
\underline{\kappa}(A_0)^2
\geq \frac{L_2^{H_0}}{15} f(\lambda) \underline{\kappa}(A_0)
\geq \frac{L_2^{H_0}}{{15\cdot 32}} \bar{\gamma} \underline{\kappa}(A_0)
\geq \frac{\abs{\pa_{pp} H}}{{15\cdot 32}} \abs{\pa_{xx} H}
\]
and so we obtain $\underline{\kappa}(A_0) \geq \frac{1}{{4\sqrt{30}}} \sqrt{\abs{\pa_{pp} H} \abs{\pa_{xx} H}}$. 

Moreover, \eqref{ineq:theirs} implies that $\underline{\kappa}(A_0) \geq L_2^{H_0}$ and so we get 
\begin{align}\label{ineq:last12}
\underline{\kappa}(A_0) \geq \frac{1}{2} L_2^{H_0} + \frac{1}{{8\sqrt{30}}} \sqrt{\abs{\pa_{pp} H} \abs{\pa_{xx} H}}.
\end{align}

To summarize, the assumptions of \cite[Theorem 7.1]{MouZha:2024} imply \eqref{ineq:last11} and \eqref{ineq:last12} which in turn imply that 
\footnotesize
\[
\underline{\kappa}(A_0) \geq \frac{1}{{8\sqrt{30}} + \sqrt{K_H}}\max\left\{ \(\frac72+\frac{\sqrt{K_H}}{2}\) L_2^{H_0} + \sqrt{\abs{\pa_{pp}H} \abs{\pa_{xx}H_0} }; \({\frac32} + f(\lambda)\)L_2^{H_0}\right\}.
\]
\normalsize
This, aside from the constant of $\frac{1}{{8\sqrt{30}} + \sqrt{K_H}}$ in front, is the exact assumption \eqref{ineq:ours} of our Proposition \ref{prop:last}.
\end{rmk}

{\bf Acknowledgements.} The authors are grateful to Wilfrid Gangbo for valuable remarks and constructive comments. MB's work was supported by the National Science Foundation Graduate Research Fellowship under Grant No. DGE-1650604 and by the Air Force Office of Scientific Research under Award No. FA9550-18-1-0502. ARM has been partially supported by the EPSRC New Investigator Award ``Mean Field Games and Master equations'' under award no. EP/X020320/1 and by the King Abdullah University of Science and Technology Research Funding (KRF) under award no. ORA-2021-CRG10-4674.2. Both authors acknowledge the partial support of the Heilbronn Institute
for Mathematical Research and the UKRI/EPSRC Additional Funding Programme for Mathematical Sciences through the focused research grant ``The master equation in Mean Field Games''.

\medskip
\medskip

{\bf Data availability statement.} No data was generated for the purposes of this research.

{\bf Declarations.} The authors declare that there are no conflicts or competing interests. 

\bibliographystyle{alpha}
\bibliography{MeanField}

\end{document}